\documentclass[reqno,a4paper, 12pt]{amsart}

\usepackage[a4paper=true,pdfpagelabels]{hyperref}
\usepackage{graphicx}

\usepackage[ansinew]{inputenc}
\usepackage{amsfonts,epsfig}
\usepackage{latexsym}
\usepackage{amsmath}
\usepackage{amssymb}

\newtheorem{theorem}{Theorem}
\newtheorem{lemma}[theorem]{Lemma}

\newtheorem{proposition}[theorem]{Proposition}

\theoremstyle{definition}

\theoremstyle{remark}

\numberwithin{equation}{section}

\newcommand{\D}{\mathbb{D}}
\newcommand{\DD}{\widehat{\mathcal{D}}}
\newcommand{\N}{\mathbb{N}}
\newcommand{\R}{\mathbb{R}}

\renewcommand{\phi}{\varphi}

\newcommand{\T}{\mathbb{T}}

\def\a{\alpha}       \def\b{\beta}        \def\g{\gamma}
           
     \def\om{\omega}      
       \def\t{\theta}       
                  \def\z{\zeta}
                  \def\vp{\varphi}

\def\R{{\mathcal R}}

\def\I{{\mathcal I}}

\newcommand{\ho}{\widehat{\omega}}

\renewcommand{\H}{\mathcal{H}}

\begin{document}

\title[ On the boundedness of  Bergman projection]{On the boundedness of  Bergman projection}

\keywords{Bergman space, Bergman Projection, equivalent norms, doubling weights.}

\thanks{This research was supported in part by the Ram\'on y Cajal program
of MICINN (Spain); by Ministerio de Edu\-ca\-ci\'on y Ciencia, Spain, projects
MTM2011-25502 and MTM2011-26538;  by   La Junta de Andaluc{\'i}a, (FQM210) and
(P09-FQM-4468);  by Academy of Finland project no. 268009,  by V\"ais\"al\"a Foundation of Finnish Academy of Science and Letters, and by Faculty of Science and Forestry of University of Eastern Finland project no. 930349.
}

\author{Jos\'e \'Angel Pel\'aez}
\address{Departamento de An\'alisis Matem\'atico, Universidad de M\'alaga, Campus de
Teatinos, 29071 M\'alaga, Spain} \email{japelaez@uma.es}

\author{Jouni R\"atty\"a}
\address{University of Eastern Finland, P.O.Box 111, 80101 Joensuu, Finland}
\email{jouni.rattya@uef.fi}
\date{\today}

%%%%%%%%%%%%%%%%%%%%%%%%%%%%%
%%%% ----  ABSTRACT ---- %%%%
%%%%%%%%%%%%%%%%%%%%%%%%%%%%%

\begin{abstract}

The main  purpose of this survey is to gather results on the boundedness of the Bergman projection. First, we shall
go over some equivalent norms on weighted Bergman spaces $A^p_\om$ which are useful in the study of this question.
 In particular,  we shall focus  on  a decomposition norm theorem for
 radial weights~$\omega$
with the doubling property $\int_{r}^1\om(s)\,ds\le C\int_{\frac{1+r}{2}}^1\om(s)\,ds$.
\end{abstract}

\maketitle
\section{Introduction}
Let $\H(\D)$ be the space  of all analytic functions in the unit disc $\D=\{z:|z|<1\}$.
If $0<r<1\,$ and $f\in \H (\D)$, set
    \begin{equation*}
    \begin{split}
    M_p(r,f)&=\left(\frac{1}{2\pi}\int_{0}^{2\pi} |f(re^{it})|^p\,dt\right)^{1/p},\quad
    0<p<\infty,\\
    M_\infty(r,f)&=\sup_{|z|=r}|f(z)|.
    \end{split}
    \end{equation*}
For $0<p\le \infty $, the Hardy space $H^p$ consists of functions
$f\in \H(\mathbb D)$ such that $\Vert f\Vert _{H^p}=
\sup_{0<r<1}M_p(r,f)<\infty$.
A  nonnegative integrable function $\om$ on the unit disc $\D$ is called a weight.
 It is radial if $\omega(z)=\omega(|z|)$ for all $z\in\D$.
 For
$0<p<\infty$ and a weight $\omega$, the weighted Bergman
space $A^p_\omega$ is the space of  $f\in\H(\D)$ for
which
    $$
    \|f\|_{A^p_\omega}^p=\int_\D|f(z)|^p\omega(z)\,dA(z)<\infty,
    $$
where $dA(z)=\frac{dx\,dy}{\pi}$ is the normalized
Lebesgue area measure on $\D$. That is, $A^p_\om=L^p_\om\cap \H(\D)$ where $L^p_\om$ is the corresponding
weighted Lebesgue space.
As usual, we write~$A^p_\alpha$ for the standard weighted
Bergman space induced by
the  radial weight $(1-|z|^2)^\alpha$, where
$-1<\alpha<\infty$ \cite{DurSchus,HKZ,Zhu}.  We denote $dA_\a=(\a+1)(1-|z|^2)^\alpha\,dA(z)$ and $\om(E)=\int_{E}\om(z)\,dA(z)$ for short.
The Carleson square $S(I)$ based on an
interval $I\subset\T$ is the set $S(I)=\{re^{it}\in\D:\,e^{it}\in I,\,
1-|I|\le r<1\}$, where $|E|$ denotes the Lebesgue measure of $E\subset\T$. We associate to
each $a\in\D\setminus\{0\}$ the interval
$I_a=\{e^{i\t}:|\arg(a e^{-i\t})|\le\frac{1-|a|}{2}\}$, and denote
$S(a)=S(I_a)$.
\medskip \par
If the norm convergence in the Bergman space $A^2_\om$ implies the uniform convergence on compact subsets, then
 the point
evaluations $L_ z$ (at the point $z\in \D$) are bounded linear
functionals on $A^2_\om$. Therefore, there are reproducing kernels
$B^\om_{z}\in A^2_\om$ with $\|L_ z\|=\|B^\om_{z}\|_{A^2_\om}$ such that
\begin{displaymath}
L_ z f=f(z)=\langle f, B^\om_{z}\rangle_{A^2_\om} =\int_{\D} f(\z)\,\overline{B^\om_{z}(\z)}\,\om(\z)\,dA(\z), \quad f\in A^2_\om.
\end{displaymath}
Since $A^2_\om$ is a closed subspace of $L^2_\om$,  we may consider the orthogonal Bergman projection $P_\om$ from $L^2_\om$ to $A^2_\om$,
that is usually called the
 Bergman projection. It is  the integral operator
    \begin{equation*}\label{intoper}
    P_\om(f)(z)=\int_{\D}f(\z)\overline{B^\om_{z}(\z)}\,\om(\z)dA(\z).
    \end{equation*}
\par The main purpose of these lectures is to gather results on the inequality
    \begin{equation}\label{twoweight}
    \|P_\om(f)\|_{L^p_v}\le C\|f\|_{L^p_v}.
    \end{equation}
We shall also provide a collection of equivalent norms on $A^p_\om$ which have been used to study this problem.
A solution  for \eqref{twoweight} is known for the class of standard weights $\om(z)=(1-|z|^2)^\alpha$ and $1<p<\infty$;
 \begin{equation*}\label{bpa}
    P_\a(f)(z)=(\a+1)\int_\D \frac{f(\z)(1-|\z|^2)^\a}{(1-z\overline{\z})^{2+\a}}\,dA(\z),\quad \a>-1,
    \end{equation*}
    is bounded on $L^p_v$ if and only if  $\frac{v(z)}{(1-|z|^2)^\a}$ belongs to the Bekoll\'e-Bonami class
$B_{p}(\a)$ \cite{BB,B}. We remind the reader that $v\in B_{p}(\a)$ if
\begin{equation}\label{eq:BB}
    \begin{split}
    &B_{p,\alpha}(v)=\sup_{I\subset\T}\frac{\left(\int_{S(I)}v(z)\,dA_\a(z)\right)
    \left(\int_{S(I)}v(z)^{\frac{-p'}{p}}\,dA_{\alpha}(z)\right)^{\frac{p}{p'}}}
    {A_\a(S(I))^p}<\infty.
    \end{split}
    \end{equation}
It is worth mentioning that the above result remains true replacing $P_\a$ by its sublinear positive counterpart
 \begin{equation*}\label{bpa+}
    P^+_\a(f)(z)=(\a+1)\int_\D \frac{|f(\z)|(1-|\z|^2)^\a}{|1-z\overline{\z}|^{2+\a}}\,dA(\z).
  \end{equation*}
Roughly speaking, this means that cancellation does not play an essential role in this question.
\medskip\par The situation is completely different when $\om$ is not a standard  weight, because of the lack of explicit expressions for
the Bergman reproducing kernels $B^\om_z$. If $\om$ is a admissible  radial weight, then the normalized monomials
$\frac{z^n}{\sqrt{2\int_0^1 r^{2n+1}\om(r)\,dr}},\, n\in\N\cup\{0\},$
form the standard  orthonormal basis of $A^2_\om$ and then \cite[Theorem~4.19]{Zhu} yields
\begin{equation}\label{repker}
B^\om_z(\z)=\sum_{n=0}^\infty\frac{(\z\bar{z})^n}{2\int_0^1 r^{2n+1}\om(r)\,dr},\quad z,\z\in\D.
\end{equation}
\par This formula and a decomposition norm theorem has been used recently in order to obtain precise estimates for the $L^p_v$-integral of~$B^\om_{z}$, see Theorem~\ref{th:kernelstimate} below. This is a key to tackle the two weight inequality \eqref{twoweight}  when $\om$ and $v$ belong to a certain class of radial weights \cite{PelRatproj}.
\par If $\om$ is not necessarily radial, the theory of weighted Bergman spaces is at its early
stages, and plenty of essential properties such as the density of polynomials (polynomials
may not be dense in $A^p_\om$ if $\om$ is not radial,  \cite[Section $1.5$]{PelRat} or \cite[p.~138]{DurSchus}) have not been described yet. Because of this fact, from now on we shall be mainly focused on Bergman spaces induced by radial weights.

Throughout the paper  $\frac{1}{p}+\frac{1}{p'}=1$. Further, the letter $C=C(\cdot)$ will denote an
absolute constant whose value depends on the parameters indicated
in the parenthesis, and may change from one occurrence to another.
We will use the notation $a\lesssim b$ if there exists a constant
$C=C(\cdot)>0$ such that $a\le Cb$, and $a\gtrsim b$ is understood
in an analogous manner. In particular, if $a\lesssim b$ and
$a\gtrsim b$, then we will write $a\asymp b$.

\section{Background on radial weights}

We shall write $\DD$ for the class of radial weights such that $\widehat{\om}(z)=\int_{|z|}^1\om(s)\,ds$ is doubling, that is, there exists $C=C(\om)\ge1$ such that $\widehat{\om}(r)\le C\widehat{\om}(\frac{1+r}{2})$ for all $0\le r<1$.
We call a radial weight $\om$ regular, denoted by $\om\in\R$, if $\om\in\DD$
and  $\om(r)$ behaves as its integral average over $(r,1)$, that is,
    \begin{equation*}
    \om(r)\asymp\frac{\int_r^1\om(s)\,ds}{1-r},\quad 0\le r<1.
    \end{equation*}
    As to concrete examples, we mention that every standard weight as well as  those given in
\cite[(4.4)--(4.6)]{AS} are regular.
It is clear that  $\om\in\R$ if and only if for each $s\in[0,1)$ there exists a
constant $C=C(s,\omega)>1$ such that
    \begin{equation}\label{eq:r2}
    C^{-1}\om(t)\le \om(r)\le C\om(t),\quad 0\le r\le t\le
    r+s(1-r)<1,
    \end{equation}
and 
\begin{equation*}
   \frac{\int_r^1\om(s)\,ds}{1-r}\lesssim \om(r),\quad0\le r<1.
    \end{equation*}
The definition of regular weights used here is slightly more general than that in \cite{PelRat}, but the main
 properties  are essentially the same by Lemma~\ref{Lemma:replacement-Lemmas-Memoirs} below and \cite[Chapter~1]{PelRat}.
 \par A radial continuous weight $\om$ is called rapidly increasing, denoted by $\om\in\I$, if
    \begin{equation*}
    \lim_{r\to 1^-}\frac{\int_r^1\om(s)\,ds }{\om(r)(1-r)}=\infty.
    \end{equation*}
It follows from \cite[Lemma~1.1]{PelRat} that $\I\subset\DD$.
Typical examples of rapidly increasing
weights are
    \begin{equation*}\label{eq:def-of-v_alpha}
    v_\a(r)=\left((1-r)\left(\log\frac{e}{1-r}\right)^\a\right)^{-1},\quad 1<\a<\infty.\index{$v_\a(r)$}
    \end{equation*}
Despite their name,  rapidly increasing weights may admit a
strong oscillatory behavior. Indeed,  the weight
    \begin{equation*}\label{pesomalo1}
    \omega(r)=\left|\sin\left(\log\frac{1}{1-r}\right)\right|v_\alpha(r)+1,\quad
    1<\alpha<\infty,
    \end{equation*}
belongs to $\I$ but it does not satisfy
\eqref{eq:r2} \cite[p.~7]{PelRat}.
\par A radial continuous weight $\om$ is called rapidly decreasing %, denoted by $\om\in\RD$,
if  $
    \lim_{r\to 1^-}\frac{\int_r^1\om(s)\,ds }{\om(r)(1-r)}=0.$ 
The exponential type weights $
    \om_{\gamma,\alpha}(r)=(1-r)^{\gamma}\exp
    \left(\frac{-c}{(1-r)^\alpha}\right), \,\gamma\ge0,\,
    \alpha,c>0,
    $
are  rapidly decreasing.
For further information on these classes see~\cite[Chapter~1]{PelRat} and the references therein.

 The following characterizations of the class $\DD$  will be frequently used from here on.
\begin{lemma}\label{Lemma:replacement-Lemmas-Memoirs}
Let $\om$ be a radial weight. Then the following conditions are equivalent:
\begin{itemize}
\item[\rm(i)] $\om\in\DD$;
\item[\rm(ii)] There exist $C=C(\om)>0$ and $\b_0=\b_0(\om)>0$ such that
    \begin{equation*}\label{Eq:replacement-Lemma1.1}
    \begin{split}
    \widehat{\om}(r)\le C\left(\frac{1-r}{1-t}\right)^{\b}\widehat{\om}(t),\quad 0\le r\le t<1,
    \end{split}
    \end{equation*}
for all $\b\ge\b_0$;
\item[\rm(iii)] There exist $C=C(\om)>0$ and $\gamma_0=\gamma_0(\om)>0$ such that
    \begin{equation*}\label{Eq:replacement-Lemma1.2}
    \begin{split}
    \int_0^t\left(\frac{1-t}{1-s}\right)^\g\om(s)\,ds
    \le C\widehat{\om}(t),\quad 0\le t<1,
    \end{split}
    \end{equation*}
for all $\g\ge\g_0$;
\item[\rm(iv)] There exists $C=C(\om)>0$ such that
    \begin{equation*}
    \begin{split}
    \int_0^ts^{\frac1{1-t}}\om(s)\,ds
    \le C\widehat{\om}(t),\quad 0\le t<1.
    \end{split}
    \end{equation*}
\item[\rm(v)] There exists $C=C(\om)>0$ such that
    $$
    \widehat{\om}(r)\le Cr^{-\frac{1}{1-t}}\widehat{\om}(t),\quad 0\le r\le t<1.
    $$
\item[\rm(vi)] The asymptotic equality
    $$
    \om_x=\int_0^1s^x\om(s)\,ds\asymp\widehat{\om}\left(1-\frac1x\right),
    $$
is valid for any  $x\ge 1$.
\end{itemize}
\end{lemma}

\begin{proof}
We are going to prove (i)$\Leftrightarrow$(ii)$\Leftrightarrow$(iii)$\Rightarrow$(iv)$\Rightarrow$(v)$\Rightarrow$(i) and (iv)$\Leftrightarrow$(vi).
\par Let $\om\in\DD$. If $0\le r\le t<1$ and $r_n=1-2^{-n}$ for all $n\in\N\cup\{0\}$, then there exist $k$ and $m$ such that $r_k\le r<r_{k+1}$ and $r_m\le t<r_{m+1}$. Hence
    \begin{equation*}
    \begin{split}
    \widehat{\om}(r)&\le\widehat{\om}(r_k)
    \le C\widehat{\om}(r_{k+1})
    \le\cdots
    \le C^{m-k+1}\widehat{\om}(r_{m+1})
    \le C^{m-k+1}\widehat{\om}(t)\\
    &=C^22^{(m-k-1)\log_2 C}\widehat{\om}(t)\le C^2\left(\frac{1-r}{1-t}\right)^{\log_2 C}\widehat{\om}(t),\quad 0\le r\le t<1.
    \end{split}
    \end{equation*}
On the other hand, it is clear that (ii) implies that $\om\in\DD$. So, we have proved  (i)$\Leftrightarrow$(ii).

If (ii) is satisfied and $\gamma>\b$, then
    \begin{equation*}
    \begin{split}
    \int_0^t\left(\frac{1-t}{1-s}\right)^\g\om(s)\,ds
    &\le C^\frac{\g}{\b}\int_0^t\left(\frac{\widehat{\om}(t)}{\widehat{\om}(s)}\right)^\frac{\g}{\b}\om(s)\,ds
    =C^\frac{\g}{\b}\widehat{\om}(t)^\frac{\g}{\b}\int_0^t\frac{\om(s)}{\left(\widehat{\om}(s)\right)^\frac{\g}{\b}}\,ds\\
    &\le\frac{\b}{\g-\b}C^\frac{\g}{\b}\widehat{\om}(t),\quad 0\le t<1.
    \end{split}
    \end{equation*}
Conversely, if (iii) is satisfied, then an integration by parts yields
    \begin{equation*}
    \begin{split}
    C\widehat{\om}(t)&\ge\int_0^t\left(\frac{1-t}{1-s}\right)^\g\om(s)\,ds\\
    &=-\widehat{\om}(t)+(1-t)^\g\widehat{\om}(0)+\gamma(1-t)^\gamma\int_0^t\frac{\widehat{\om}(s)}{(1-s)^{\gamma+1}}\,ds\\
    &\ge-\widehat{\om}(t)+(1-t)^\g\widehat{\om}(0)+\gamma(1-t)^\gamma\widehat{\om}(r)\int_0^r\frac{ds}{(1-s)^{\gamma+1}}\\
    &=-\widehat{\om}(t)+(1-t)^\g(\widehat{\om}(0)-\widehat{\om}(r))+\left(\frac{1-t}{1-r}\right)^\gamma\widehat{\om}(r)\\
    &\ge\left(\frac{1-t}{1-r}\right)^\gamma\widehat{\om}(r)-\widehat{\om}(t),\quad 0\le r\le t<1,
    \end{split}
    \end{equation*}
    therefore (ii)$\Leftrightarrow$(iii).

The proof of \cite[Lemma~1.3]{PelRat} shows that (iii) implies (iv), we include a proof  for the sake of completeness.
 A simple calculation shows that for all $s\in (0,1)$ and $x>1$,
    $$
    s^{x-1}(1-s)^\gamma\le\left(\frac{x-1}{x-1+\gamma}\right)^{x-1}\left(\frac{\gamma}{x-1+\gamma}\right)^\gamma
    \le\left(\frac{\gamma}{x-1+\gamma}\right)^\gamma.
    $$
 Therefore (iii), with
$t=1-\frac{1}{x}$, yields
   \begin{equation*}
    \begin{split}
    \int_0^{1-\frac{1}{x}} s^{x}\om(s)\,ds
    &\le\left(\frac{\gamma x}{x-1+\gamma}\right)^\gamma\int_0^{1-\frac{1}{x}}\frac{\om(s)}{x^\gamma(1-s)^\gamma}s\,ds
    \\ & \lesssim
    \int_{1-\frac{1}{x}}^1\om(s)\,ds,\quad x>1,
     \end{split}
    \end{equation*}
 which is equivalent to (iv).
\par On the other hand, if (iv) is satisfied and $0\le r\le t<1$, then an integration by parts yields
    \begin{equation*}
    \begin{split}
    C\widehat{\om}(t)&\ge\int_0^ts^\frac{1}{1-t}\om(s)\,ds
    =-\widehat{\om}(t)t^\frac{1}{1-t}+\frac{1}{1-t}\int_0^t\widehat{\om}(s)s^\frac{t}{1-t}\,ds\\
    &\ge-\widehat{\om}(t)t^\frac{1}{1-t}+\frac{1}{1-t}\int_0^r\widehat{\om}(s)s^\frac{t}{1-t}\,ds\\
    &\ge-\widehat{\om}(t)t^\frac{1}{1-t}+\frac{\widehat{\om}(r)}{1-t}\int_0^rs^\frac{t}{1-t}\,ds
    =-\widehat{\om}(t)t^\frac{1}{1-t}+r^\frac{t}{1-t}\widehat{\om}(r),
    \end{split}
    \end{equation*}
and thus
    $$
    r^\frac{1}{1-t}\widehat{\om}(r)\le\left(C+t^\frac1{1-t}\right)\widehat{\om}(t),\quad 0\le r\le t<1.
    $$
This implies (v), and by choosing $t=\frac{1+r}{2}$ in (v), we deduce $\om\in\DD$ .
 Finally, it is clear that (iv) is equivalent to (vi).
\end{proof}
\section{Equivalent norms}
In this section we shall present several equivalent norms on weighted Bergman spaces. In particular
we shall give a detailed proof of a decomposition norm theorem for $A^p_\om$ when  $\om\in\DD$ and $1<p<\infty$.
\par It is well-known that a choice of an appropriate norm is often a key step when solving
 a problem on a space of analytic functions.
For instance, in the study of the integration operators
\begin{displaymath}
    T_g(f)(z)=\int_{0}^{z}f(\zeta)\,g'(\zeta)\,d\zeta,\quad
    z\in\D,\quad g\in\H(\D),
    \end{displaymath}
one wants to get rid of the integral symbol, so one looks  for   norms
in terms of the first derivative.
It is worth mentioning that the operator $T_g$  began to be extensively
studied after the appearance of the works by Aleman, Cima and
Siskakis~\cite{AC,AS}. A description of its resolvent set on Hardy and standard Bergman spaces  is
strongly connected with the classical theory of the Muckenhoupt weights and the Bekoll\'e-Bonami weights
\cite{AlCo,AlPe}.
\subsection{Norms in terms of the derivative}
Following Siskakis \cite{Si},
for a radial weight $\om$ we define its distortion function by

    $$
    \psi_{\om}(z)=\frac{1}{\om(|z|)}\int_{|z|}^1\om(s)\,ds,\quad
    z\in\D.
    $$
For a large class of radial weights, which includes any
differentiable decreasing weight and all the standard ones, the
most appropriate way to obtain a useful norm involving the first
derivative is to establish a kind of
Littlewood-Paley type formula
\cite[Theorem $1.1$]{PavP}.
\begin{theorem}\label{PavP}
Suppose that  $\om$ is a radial differentiable weight, and there is $L>0$
such that
\begin{equation*}\label{eq:L}
\sup_{0<r<1}\frac{\om'(r)}{\om(r)^2}\int_r^1 \om(x)\,dx\ \le L.
\end{equation*}
 Then, for each $p\in (0,\infty)$
\begin{displaymath}
\int_{\D}| f(z)|^p \om(z)\,dA(z)\asymp
|f(0)|^p+\int_{\D}|f'(z)|^p\,\psi^p_ \om(z)\,\om(z)\,dA(z),\quad f\in \H(\D).
\end{displaymath}
\end{theorem}
If $\om\in\I$ and $p\neq 2$, a result analogous to Theorem~\ref{PavP} cannot be obtained  in
general \cite[Proposition $4.2$]{PelRat}.
\begin{proposition}\label{pr:NOL-P}
Let $p\ne2$. Then there exists $\om\in\I$ such that, for any
function $\vp:[0,1)\to(0,\infty)$, the relation
\begin{equation*}\label{eq:NOL-P}
    \|f\|^p_{A^p_{\om}}\asymp
    \int_\D|f'(z)|^p\varphi(|z|)^p\om(z)\,dA(z)+|f(0)|^p
    \end{equation*}
can not be valid for all $f\in\H(\D)$.
\end{proposition}
\par As for a Littlewood-Paley formula for $A^p_\om$,
the following result  was proved in \cite[Theorem $3.1$]{AlCo}.
\begin{theorem}\label{PavP}
Suppose that $\om$ is a  weight such that $\frac{\om(z)}{(1-|z|)^\eta}$ satisfies the Bekoll\'e-Bonami condition $B_{p_0}(\eta)$
for some $p_0>0$ and some $\eta>-1$.
 Then, for each $p\in (0,\infty)$
\begin{equation}\label{LPbek}
\int_{\D}| f(z)|^p \om(z)\,dA(z)\asymp
|f(0)|^p+\int_{\D}|f'(z)|^p\,(1-|z|)^p\,\om(z)\,dA(z),\quad f\in \H(\D).
\end{equation}
\end{theorem}
We remark that whenever $\om\in C^1(\D)$  and
$$(1-|z|)|\nabla \om(z)|\lesssim \om(z),\quad z\in\D,$$ \eqref{LPbek} is equivalent to a Bekoll\'e-Bonami condition \cite[Theorem $3.1$]{AlCo}.

\medskip\par Now, we consider the non-tangential approach regions
    \begin{equation*}\label{eq:gammadeuintro}
    \Gamma(\z)=\left\{z\in \D:\,|\t-\arg
    z|<\frac12\left(1-\frac{|z|}{r}\right)\right\},\quad
    \z=re^{i\theta}\in\D\setminus\{0\}
    \end{equation*}
and the related tents $T(z)=\left\{\z\in \D:\,z\in\Gamma(\z)\right\}$.
\par Whenever $\om$ is a radial weight,  $A^p_\om$  can be equipped with other norms which are inherited
from the classical Fefferman-Stein
estimate~\cite{FC} and the Hardy-Stein-Spencer
identity~\cite{Garnett1981} for the
 $H^p$-norm. Here
$
    \omega^\star(z)=\int_{|z|}^1\omega(s)\log\frac{s}{|z|}s\,ds,
    $\, $z\in\D\setminus\{0\}$.

\begin{theorem}\label{ThmLittlewood-Paley}
Let $0<p<\infty$, $n\in\N$ and $f\in\H(\D)$, and let $\omega$ be a
radial weight. Then
    \begin{equation*}\label{HSB}
    \|f\|_{A^p_\omega}^p=p^2\int_{\D}|f(z)|^{p-2}|f'(z)|^2\omega^\star(z)\,dA(z)+\omega(\D)|f(0)|^p,
    \end{equation*}
    and
\begin{equation*}\label{normacono}
    \begin{split}
    \|f\|_{A^p_\omega}^p&\asymp\int_\D\,\left(\int_{\Gamma(u)}|f^{(n)}(z)|^2
    \left(1-\left|\frac{z}{u}\right|\right)^{2n-2}\,dA(z)\right)^{\frac{p}2}\omega(u)\,dA(u)\\
    &\quad+\sum_{j=0}^{n-1}|f^{(j)}(0)|^p,
    \end{split}
    \end{equation*}
where the constants of comparison depend only on $p$, $n$ and $\om$. In particular,
    \begin{equation*}\label{eq:LP2}
    \|f\|_{A^2_\omega}^2=4\|f'\|_{A^2_{\omega^\star}}^2+\omega(\D)|f(0)|^2.
    \end{equation*}
\end{theorem}
\par Next, we  present an equivalent norm for weighted Bergman spaces which has been very recently used
to describe the $q$-Carleson mesures for $A^p_\om$ when $\om\in\DD$ \cite{PelRatMathAnn}.

Let $f\in\H(\D)$, and define the \emph{non-tangential maximal
function} of $f$ in the (punctured) unit disc by
    $$\index{$N(f)$}\index{non-tangential maximal function}
    N(f)(u)=\sup_{z\in\Gamma(u)}|f(z)|,\quad
    u\in\D\setminus\{0\}.
    $$

\begin{lemma}\label{le:funcionmaximalangular}\cite[Lemma $4.4$]{PelRat}
Let $0<p<\infty$ and let $\om$ be a radial weight. Then there
exists a constant $C>0$ such that
    $$
    \|f\|^p_{A^p_\om}\le\|N(f)\|^p_{L^p_\om}\le C\|f\|^p_{A^p_\om}
    $$
for all $f\in\H(\D)$.\index{$N(f)$}\index{non-tangential maximal
function}
\end{lemma}

\begin{proof}
It follows from \cite[Theorem~3.1 on p.~57]{Garnett1981} that
there exists a constant $C>0$ such that the classical
non-tangential maximal function
    $$
    f^\star(\z)=\sup_{z\in\Gamma(\z)}|f(z)|\index{$f^\star$},\quad
    \z\in\T,
    $$
satisfies
    \begin{equation*}\label{Eq:Maximal}\index{$f^\star$}
    \|f^\star\|^p_{L^p(\T)}\le C\|f\|^p_{H^p}
    \end{equation*}
for all $0<p<\infty$ and $f\in\H(\D)$. Therefore
    \begin{equation*}\index{$N(f)$}\index{non-tangential maximal function}
    \begin{split}
    \|f\|^p_{A^p_\om}&\le\|N(f)\|^p_{L^p_\om}=\int_\D(N(f)(u))^p\om(u)\,dA(u)\\
    &=\int_0^1\om(r)r\int_{\T}((f_r)^\star(\z))^p\,|d\z|\,dr\\
    &\le C\int_0^1\om(r)r\int_{\T}f(r\z)^p\,|d\z|\,dr
    =C\|f\|^p_{A^p_\om},
    \end{split}
    \end{equation*}
and the assertion is proved.
\end{proof}
\subsection{Decomposition norm theorems}
The main purpose   of this section is to extend \cite[Theorem $4$]{PelRathg} to the case of when $\om\in\DD$.
Decomposition norm theorems have been obtained previously in \cite{MatelPavstu,Pav86,Pav87} for several type of mixed norm spaces.
For $0<p\le \infty$,
$0<q<\infty$,  and a radial weight~$\om$, the
 mixed norm space $H(p,q,\om)$ consists of those
$g\in\H(\D)$ such that
    \begin{equation*}
    \left\|g\right\|^q_{H(p,q,\om)}=\int_0^1 M^q_p(r,g)\om(r)\,dr<\infty.
    \end{equation*}
If in addition $-\infty<\b<\infty$, we will denote $g\in
H(p,\infty,\widehat{\om}^\b)$, whenever
    \begin{equation*}
    \left\|g\right\|_{H(p,\infty,\widehat{\om}^\b)}=\sup_{0<r<1}M_p(r,g)\ho(r)^\b<\infty.
    \end{equation*}
 It is clear that
$H(p,p,\om)=A^p_\om$. The mixed norm spaces play an essential role
in the closely related question of studying the coefficient
multipliers and the generalized Hilbert operator
\begin{equation*}
    \mathcal{H}_g(f)(z)=\int_0^1f(t)g'(tz)\,dt,\quad g\in\H(\D),
    \end{equation*}
 on Hardy and weighted Bergman spaces~\cite{ArJeVu,GPPS14,PelRathg}.
\par In order to give the precise statement of the main result of this section, we need
to introduce some more notation. To do this, let $\om$ be a radial weight such
that $\int_0^1 \om(r)\,dr=1$. For each $\a>0$ and
$n\in\N\cup\{0\}$, let $r_n=r_n(\om,\a)\in[0,1)$ be defined by
    \begin{equation}\label{rn}
    \widehat{\om}(r_n)=\int_{r_n}^1\om(r)\,dr=\frac{1}{2^{n\a}}.
    \end{equation}
Clearly, $\{r_n\}_{n=0}^\infty$ is an increasing sequence of
distinct points on $[0,1)$ such that $r_0=0$ and $r_n\to1^-$, as
$n\to\infty$. For $x\in[0,\infty)$, let $E(x)$ denote the integer
such that $E(x)\le x<E(x)+1$, and set
$M_n=E\left(\frac{1}{1-r_{n}}\right)$ for short. Write
    $$
    I(0)=I_{\om,\alpha}(0)=\left\{k\in\N\cup\{0\}:k<M_1\right\}
    $$
and
   \begin{equation*}
    I(n)=I_{\om,\a}(n)=\left\{k\in\N:M_n\le
    k<M_{n+1}\right\}
  \end{equation*}
for all $n\in\N$. If
$f(z)=\sum_{n=0}^\infty a_nz^n$ is analytic in~$\D,$ define the
polynomials $\Delta^{\om,\a}_nf$ by
    \[
    \Delta_n^{\om,\a}f(z)=\sum_{k\in I_{\om,\a}(n)} a_kz^k,\quad n\in\N\cup\{0\}.
    \]
\begin{theorem}\label{th:dec1}
Let $1<p<\infty$, $0<\a<\infty$ and $\om\in\DD$ such that
$\int_0^1\om(r)\,dr=1$, and let $f\in\H(\D)$.
\begin{itemize}
\item[\rm(i)] If $0<q<\infty$, then $f\in H(p,q,\om)$ if and only
if
    \begin{equation*}
    \sum_{n=0}^\infty 2^{-n\alpha}\|\Delta^{\om,\a}_n f\|_{H^p}^q<\infty.
    \end{equation*}
Moreover,
    \[
    \|f\|_{H(p,q,\om)}\asymp \bigg(\sum_{n=0}^\infty 2^{-n\a}
    \|\Delta^{\om,\a}_n f\|_{H^p}^q\bigg)^{1/q}.
    \]
\item[\rm(ii)] If $0<\beta<\infty$, then $f\in
H(p,\infty,\ho^\b)$ if and only if
    $$
    \sup_n 2^{-n\a\beta} \| \Delta^{\om,\a}_n f\|_{H^p}<\infty.
    $$
Moreover,
    \[
    \|f\|_{H(p,\infty,\ho^\b)}\asymp \sup_n2^{-n\a\beta}\|
    \Delta^{\om,\a}_n f\|_{H^p}.
    \]
\end{itemize}
\end{theorem}
\par The  proof of Theorem \ref{th:dec1} follows that of \cite[Theorem $4$]{PelRat}, and it only distinguishes from it
because of the technicalities of broadening the class $\R\cup\I$ to $\DD$.
Some  previous results are needed. Recall that a
function $h$ is called essentially decreasing if there exists a
positive constant $C=C(h)$ such that $h(x)\le C h(y)$ whenever
$y\le x$. Essentially increasing functions are defined in an
analogous manner.

\begin{lemma}\label{le:serie}
Let $\om\in\DD$ such that $\int_0^1 \om(r)\,dr=1$. For each
$\a>0$ and $n\in\N\cup\{0\}$, let $r_n=r_n(\om,\a)\in[0,1)$ be
defined by \eqref{rn}. Then the following assertions hold:
\begin{enumerate}
\item[\rm(i)] For each $\gamma>0$, there exists
$C=C(\a,\gamma,\om)>0$ such that
    \begin{equation}\label{serie}
    \eta_\gamma(r)=\sum_{n=0}^\infty
    2^{n\gamma}r^{M_n}\le
    C\,\widehat{\om}(r)^{-\frac{\gamma}{\a}},\quad
    0\le r<1.
    \end{equation}
    \item[\rm(ii)] For each $0<\b<1$, there exists
$C=C(\a,\b,\om)>0$ such that
    \begin{equation}\label{momentos}
    2^{-n\a\b}\int_{0}^1
    \frac{r^{M_n}\om(r)}{\widehat{\om}(r)^{\b}}\,dr\le C\int_{0}^1
    r^{M_n}\om(r)\,dr.
    \end{equation}
\end{enumerate}
\end{lemma}

\begin{proof}
(i). We will begin with proving \eqref{serie} for $r=r_N$, where
$N\in\N$. To do this, note first that
    \begin{equation}\label{serie2}
    \begin{split}
    \sum_{n=0}^N2^{n\gamma}r_N^{M_n}
    \le \frac{2^\gamma}{2^\gamma-1}
    \,\widehat{\om}(r_N)^{-\frac{\gamma}{\a}}
    \end{split}
    \end{equation}
by \eqref{rn}. To deal with the remainder of the sum, we apply Lemma
\ref{Lemma:replacement-Lemmas-Memoirs}(ii) and \eqref{rn} to find
$\b=\beta(\om)>0$ and $C=C(\b,\om)>0$ such that
    $$
    \frac{1-r_n}{1-r_{n+j}}\ge C\left(\frac{\widehat{\om}(r_n)}{\widehat{\om}(r_{n+j})}\right)^{1/\b}
    = C2^{\frac{j\a}{\b}},\quad n,j\in\N\cup\{0\}.
    $$
This, the inequality $\log\frac1x\ge1-x$, $0<x\le1$, and
\eqref{rn} give
    \begin{equation*}
    \begin{split}
    \sum_{n=N+1}^\infty 2^{n\gamma}r_N^{M_n}
    &\le2^{N\g}\sum_{j=1}^\infty 2^{j\g}e^{-r_{N+j}\frac{1-r_N}{1-r_{N+j}}}
    \le2^{N\g}\sum_{j=1}^\infty 2^{j\g}e^{-r_2C2^{\frac{j\a}{\b}}}\\
    &=C(\b,\a,\gamma,\om)\,\widehat{\om}(r_N)^{-\frac{\gamma}{\a}}.
    \end{split}
    \end{equation*}
Since $\b=\b(\om)$, this together with \eqref{serie2} gives
\eqref{serie} for $r=r_N$, where $N\in\N$. Now, using standard arguments, it implies \eqref{serie} for any $r\in(0,1)$.

(ii). Let us write $\widetilde{\om}(r)=\frac{\om(r)}{\widehat{\om}(r)^{\b}}$. Clearly,
    \begin{equation}\label{60}
    \begin{split}
    2^{-n\a\b}\int_{0}^{r_n}
    r^{M_n}\widetilde{\om}(r)\,dr&\le
    \frac{2^{-n\a\b}}{\widehat{\om}(r_n)^{\b}} \int_{0}^{r_n}
    r^{M_n}\om(r)\,dr\le \int_{0}^1 r^{M_n}\om(r)\,dr.
    \end{split}
    \end{equation}
Moreover, \cite[Lemma $1.4$ (iii)]{PelRat} yields
    \begin{equation}\label{61}
    \begin{split}
    2^{-n\a\b}\int_{r_n}^1
    r^{M_n}\widetilde{\om}(r)\,dr
    &\le2^{-n\a\b}\widetilde{\om}(r_n)\psi_{\widetilde{\om}}(r_n)
    =\frac{2^{-n\a\b}}{1-\b}\widetilde{\om}(r_n)\psi_{\om}(r_n)\\
    &=\frac{1}{1-\b}\int_{r_n}^1\om(r)\,dr\le C(\b,\a,\om)\int_{r_n}^1 r^{M_n}\om(r)\,dr.
    \end{split}
    \end{equation}
By combining \eqref{60} and \eqref{61} we obtain (ii).
\end{proof}

We now present a result on power series with positive
coefficients. This result will play a crucial role in the proof of
Theorem~\ref{th:dec1}.

\begin{proposition}\label{pr:DecompApw}
Let $0<p,\alpha<\infty$ and $\om\in\DD$ such that $\int_0^1
\om(r)\,dr=1$. Let $f(r)=\sum_{k=0}^\infty a_k r^k$, where $a_k\ge
0$ for all $k\in\N\cup\{0\}$, and denote $t_n=\sum_{k\in
I_{\om,\a}(n)}a_k$. Then there exists a constant $C=C(p,\a,\om)>0$ such that
    \begin{equation}\label{j13}
    \frac{1}{C}\sum_{n=0}^\infty 2^{-n\a}t_n^p\le \int_{0}^1
    f(r)^p\om(r)\,dr\le C \sum_{n=0}^\infty 2^{-n\a}t_n^p.
    \end{equation}
\end{proposition}

\begin{proof}
We will use ideas from the proof of \cite[Theorem~6]{MatelPav}.
The definition \eqref{rn} yields
    \begin{equation*}
    \begin{split}
    \int_{0}^1f(r)^p\om(r)\,dr &\ge
    \sum_{n=0}^\infty\int_{r_{n+1}}^{r_{n+2}}\left(\sum_{k=0}^\infty t_k
    r^{M_{k+1}}\right)^p\om(r)\,dr\\
    & \ge \sum_{n=0}^\infty\left(\sum_{k=0}^n t_k r_{n+1}^{M_{k+1}}\right)^p\int_{r_{n+1}}^{r_{n+2}}\om(r)\,dr\\
    & \ge \left(1-\frac1{2^\a}\right)\sum_{n=0}^\infty t^p_n r_{n+1}^{pM_{n+1}}2^{(-n-1)\a}\ge C \sum_{n=0}^\infty t^p_n2^{-n\a},
    \end{split}
    \end{equation*}
where $C=C(p,\a,\om)>0$ is a constant. This gives the first
inequality in~\eqref{j13}.

To prove the second inequality in~\eqref{j13}, let first $p>1$ and
take $0<\g<\frac{\a}{p-1}$. Then H\"older's inequality gives
    \begin{equation*}\label{10.}
    \begin{split}
    f(r)^p\le \left(\sum_{n=0}^\infty t_n r^{M_{n}}\right)^p\le
    \eta_{\g}(r)^{p-1}\sum_{n=0}^\infty 2^{-n\gamma(p-1)}t_n^p
    r^{M_{n}}.
    \end{split}
    \end{equation*}
Therefore, by \eqref{serie} and \eqref{momentos} in
Lemma~\ref{le:serie} and  Lemma~\ref{Lemma:replacement-Lemmas-Memoirs}(vi) there exist
constants $C_1=C_1(\a,\g,p,\om)>0$, $C_2=C_2(\a,\g,p,\om)>0$ and
$C_3=C_3(\a,\g,p,\om)>0$ such that
    \begin{equation*}
    \begin{split}
    \int_{0}^1f(r)^p\om(r)\,dr &\le \sum_{n=0}^\infty2^{-n\gamma(p-1)}t_n^p\int_{0}^1
    r^{M_{n}}\eta_{\g}(r)^{p-1}\om(r)\,dr\\
    &\le C_1\sum_{n=0}^\infty  2^{-n\gamma(p-1)}t_n^p\int_{0}^1\frac{r^{M_{n}}\om(r)}{\widehat{\om}(r)^{\frac{\gamma(p-1)}{\a}}}\,dr\\
    &\le C_2\sum_{n=0}^\infty t^p_n\int_{0}^1 r^{M_{n}}\om(r)\,dr\\
    &\le C_3\sum_{n=0}^\infty t^p_n\,\widehat{\om}(r_n)\,dr=C_3\sum_{n=0}^\infty
    t^p_n2^{-n\alpha}.
    \end{split}
    \end{equation*}
Since $\g=\g(\a,p)$, this gives the assertion for $1<p<\infty$.
\par If  $0<p\le1$, then \begin{equation*}
    \begin{split}
    f(r)^p\le \left(\sum_{n=0}^\infty t_n r^{M_{n}}\right)^p\le
    \sum_{n=0}^\infty t_n^p
    r^{M_{n}p},
    \end{split}
    \end{equation*}
    so using  Lemma~\ref{Lemma:replacement-Lemmas-Memoirs}(vi) and (ii), there exists
a constant $C_1=C_1(\a,\g,p,\om)>0$ such that
\begin{equation*}
    \begin{split}
    \int_{0}^1f(r)^p\om(r)\,dr &\le \sum_{n=0}^\infty t^p_n\int_{0}^1 r^{pM_{n}}\om(r)\,dr
    \\
    &\le C_1\sum_{n=0}^\infty t^p_n\,\widehat{\om}(r_n)=C_1\sum_{n=0}^\infty
    t^p_n2^{-n\alpha}.
    \end{split}
    \end{equation*}
    This finishes the proof.
\end{proof}

Next, for
$g(z)=\sum_{k=0}^\infty b_k z^k\in\H(\D)$ and
$n_1,n_2\in\N\cup\{0\}$, we set
    $$
    S_{n_1,n_2}g(z)=\sum_{k=n_1}^{n_2-1}b_kz^k,\quad n_1<n_2.
    $$
The chain of inequalities
\begin{equation}\label{Eq-MaPa}
     r^{n_2}\|S_{n_1,n_2}g\|_{H^p}\le M_p(r,S_{n_1,n_2}g)\le r^{n_1}\|S_{n_1,n_2}g\|_{H^p}, \quad
     0<r<1,
    \end{equation}
follows from \cite[Lemma~3.1]{MatelPavstu}.
\begin{lemma}\label{le:4}
Let $0<p\le \infty$ and $n_1,n_2\in\N$ with $n_1<n_2$. If
$g(z)=\sum_{k=0}^{\infty} c_k z^k\in \H(\D)$, then
\begin{equation*}
\| S_{n_1,n_2}g\|_{H^p}\asymp M_p\left(1-\frac{1}{n_2}, S_{n_1,n_2}g\right).
\end{equation*}
\end{lemma}

\medskip

\noindent\emph{Proof of} Theorem~\ref{th:dec1}. (i). By the
M.~Riesz projection theorem and \eqref{Eq-MaPa},
    \begin{equation*}
    \begin{split}
    \|f\|_{H(p,q,\om)}
    &\gtrsim\sum_{n=0}^\infty \|\Delta^{\om,\a}_n f\|_{H^p}^q\int_{r_{n+1}}^{r_{n+2}}
    r^{q{M_{n+1}}}\om(r)\,dr\\
    &\asymp\sum_{n=0}^\infty \|\Delta^{\om,\a}_n f\|_{H^p}^q\int_{r_{n+1}}^{r_{n+2}}\om(r)\,dr
    \asymp\sum_{n=0}^\infty 2^{-n\a}\|\Delta^{\om,\a}_n f\|_{H^p}^q.
    \end{split}
    \end{equation*}
On the other hand, Minkowski's inequality and \eqref{Eq-MaPa} give
    \begin{equation}\label{1}
    M_p(r,f)\le\sum_{n=0}^\infty M_p(r,\Delta^{\om,\a}_nf)\le\sum_{n=0}^\infty
    r^{M_n}\|\Delta^{\om,\a}_n f\|_{H^p},
    \end{equation}
and hence Proposition~\ref{pr:DecompApw} yields
    \begin{equation*}
    \begin{split}
    \|f\|_{H(p,q,\om)}\le\int_0^1 \left( \sum_{n=0}^\infty
    r^{M_n} \|\Delta^{\om,\a}_n f\|_{H^p}\right)^q\om(r)\,dr
    \asymp\sum_{n=0}^\infty 2^{-n\a}\|\Delta^{\om,\a}_n
    f\|_{H^p}^q.
    \end{split}
    \end{equation*}

(ii). Using again the M.~Riesz projection theorem and
\eqref{Eq-MaPa} we deduce
    \[
    \sup_{0<r<1}M_p(r,f)\,\widehat{\om}(r)^\b\gtrsim
    r_{n+1}^{{M_{n+1}}}\|\Delta^{\om,\a}_nf\|_{H^p}2^{-n\a\beta},\quad
    n\in\N\cup\{0\},
    \]
and hence
    $$
    \|f\|_{H(p,\infty,\widehat{\om}^{\b})}\gtrsim\sup_n2^{-n\a\beta}
    \|\Delta^{\om,\a}_nf\|_{H^p}.
    $$
Conversely, assume that $M=\sup_n 2^{-n\a\beta}\|
\Delta^{\om,\a}_n f\|_{H^p}<\infty$. Then \eqref{1} and
Lemma~\ref{le:serie}(i) yield
    \[
    \begin{aligned}
    M_p(r,f)&\le\sum_{n=0}^\infty r^{M_n}\|\Delta^{\om,\a}_n f\|_{H^p}\le M\sum_{n=0}^\infty
    2^{n\a\beta}r^{M_n} \lesssim M
    \widehat{\om}(r)^{-\beta}.
    \end{aligned}
    \]
This finishes the proof. \hfill$\Box$
\par It is worth mentioning that Theorem~\ref{th:dec1} does not remain valid for $0<p\le 1$.
But the part that is true in this case is contained in the next result.
\begin{proposition}\label{le:de1}
Let $0<p\le 1$, $0<\a<\infty$ and $\om\in\DD$ such that
$\int_0^1\om(r)\,dr=1$.
\begin{itemize}
\item[\rm(i)] If $0<q<\infty$, then
    \[
    \|f\|_{H(p,q,\om)}\lesssim \bigg(\sum_{n=0}^\infty 2^{-n\a}
    \|\Delta^{\om,\a}_n f\|_{H^p}^q\bigg)^{1/q},\quad f\in \H(\D).
    \]
\item[\rm(ii)] If $0<\beta<\infty$, then
    \[
    \|f\|_{H(p,\infty,\ho^\b)}\lesssim \sup_n2^{-n\a\beta}\|
    \Delta^{\om,\a}_n f\|_{H^p},\quad f\in \H(\D).
    \]
\end{itemize}
\end{proposition}
 Proposition \ref{le:de1} follows from  the inequality
\begin{equation*}
    M^p_p(r,f)\le\sum_{n=0}^\infty M^p_p(r,\Delta^{\om}_nf)
    \le\sum_{n=0}^\infty
    r^{pM_n}\|\Delta^{\om}_n f\|^p_{H^p},
    \end{equation*}
\eqref{Eq-MaPa} and Proposition \ref{pr:DecompApw}. See also \cite[Lemma $8$]{PelRatproj}.

\section{Bergman Projection}
\subsection{One weight inequality}
\par The boundedness of projections on $L^p$-spaces   is an intriguing topic  which has  attracted a lot  attention
in recent years \cite{ArPau,CP2,D1,D3,HKZ,PelRatproj,ZeyTams2012,Zhu}. In fact, as far as we know, to characterize those radial weights for which $P_\om:L^p_\om\to L^p_\om$ is bounded, is still an open problem~\cite[p.~116]{D1}.
\par For the class of standard weights, the Bergman projection $P_\a$ (as well as $P^+_\alpha$)
    is bounded on $L^p_\alpha$ if and only if $1<p<\infty$ \cite[Theorem $4.24$]{Zhu}.
    As for $p=\infty$, $P_\a$ is bounded and onto
 from $L^\infty$ to $\mathcal{B}$. Here~$\mathcal{B}$ \cite[Chapter $5$]{Zhu} denotes the Bloch space that consists of $f\in\H(\D)$ such that
    $$
    \|f\|_{\mathcal{B}}=\sup_{z\in\D}|f'(z)|(1-|z|^2)+|f(0)|<\infty.
    $$ These results have  been recently extended to the class of
    regular weights \cite{PelRatproj}.
\begin{theorem}\label{theorem:projections}
Let $1<p<\infty$.
\begin{enumerate}
\item[\rm(i)] If $\om\in\R$, then $P^+_\om:L^p_\om\to L^p_\om$ is bounded. In particular, $P_\om:L^p_\om\to A^p_\om$ is bounded.
\item[\rm(ii)] If $\om\in\R$, then $P_\om:L^\infty(\D)\to \mathcal{B}$ is bounded.
\end{enumerate}
\end{theorem}

  In the original source \cite{PelRatproj}, Theorem \ref{theorem:projections} (i) is obtained as  a consequence of Theorem \ref{theorem:projections3} below.
  Here, we shall offer a simple proof of this result. Both arguments
  use strongly  precise
     $L^p$-estimates of the Bergman reproducing kernels \cite{PelRatproj}.
    \begin{theorem}\label{th:kernelstimate}
Let $0<p<\infty$, $\om\in\DD$ and $N\in\N\cup\{0\}$. Then the following assertions hold:
\begin{enumerate}
\item[\rm(i)]$\displaystyle M_p^p\left(r,\left(B^\om_a\right)^{(N)}\right)\asymp
    \int_0^{|a|r}\frac{dt}{\widehat{\om}(t)^{p}(1-t)^{p(N+1)}},\quad r,|a|\to1^-.$
\item[\rm(ii)] If $v\in\DD$, then
    \begin{equation*}\label{k1}
    \|\left(B^\om_a\right)^{(N)}\|^p_{A^p_v}
    \asymp \int_0^{|a|}\frac{\widehat{v}(t)}{\widehat{\om}(t)^{p}(1-t)^{p(N+1)}}\,dt,\quad |a|\to1^-.
    \end{equation*}
\end{enumerate}
\end{theorem}
\par We would like to mention that Theorem~\ref{PavP} and \cite[Theorem $4$]{PelRathg} play important roles in the proof of  this result.
 Besides, we use strongly Lemma~\ref{Lemma:replacement-Lemmas-Memoirs}, in particular the description of the class $\DD$  in terms of the moments of the weights
$$\int_0^1s^x\om(s)\,ds\asymp\widehat{\om}\left(1-\frac1x\right),\quad x\in[1,\infty).
    $$

Now, we offer a simple proof of the one weight inequality for regular weights.
\par \noindent\emph{Proof of Theorem~\ref{theorem:projections} (i).}
Let $1<p<\infty$ and $\om\in\R$. Let $h=\widehat{\om}^{-\frac{1}{pp'}}$, where $\frac{1}{p}+\frac{1}{p'}=1$.
Since $p>1$, \cite[Lemma~1.4(iii)]{PelRat} shows that
$h^{p'}\om$ is a weight with $\psi_{h^{p'}\om}=\frac{p}{p-1}\psi_{\om}$, and thus $h^{p'}\om\in\R$.
Since $\om\in\DD$,  by Lemma \ref{Lemma:replacement-Lemmas-Memoirs}(ii)  there
exists $\b=\b(\om)$ such that $\widehat{\om}(s)(1-s)^{-\b}$ is essentially increasing on $[0,1)$.
On the other hand, since $\om\in\R$ there
is  $\a=\a(\om)>0$ with $\a\le \b$ such that $\widehat{\om}(s)(1-s)^{-\a}$
is essentially decreasing, see
 \cite[(ii)~p.~10]{PelRat}. By using this and $h^{p'}\om\in\R$ we deduce
    \begin{equation}\label{1111}
    \int_0^r\frac{\widehat{h^{p'}\om}(s)}{\widehat{\om}(s)(1-s)}\,ds
    \asymp\int_0^r\frac{ds}{\widehat{\om}(s)^\frac1p(1-s)}
    \asymp\frac{1}{\widehat{\om}(r)^\frac1p}=h^{p'}(r),\quad r\ge\frac12.
    \end{equation}
By symmetry, a similar reasoning applies when $p'$ is replaced by $p$, and hence we may use Theorem~\ref{th:kernelstimate}(ii) and \eqref{1111} to deduce
    $$
    \int_{\D} |B^\om(z,\z)|h^{p'}(\z)\om(\z)\,dA(\z)\asymp h^{p'}(z),\quad z\in\D,
    $$
and
    $$
    \int_{\D} |B^\om(z,\z)|h^{p}(z)\om(z)\,dA(z)\asymp h^{p}(\z),\quad \z\in\D.
    $$
It follows from  Schur's test~\cite[Theorem~3.6]{Zhu} that  $P^+_\om:L^p_\om\to L^p_\om$ is bounded.
\hfill$\Box$
\medskip\par The situation is different for $\om\in\I$ because then $P^+_\om$ is not bounded on $L^p_\om$ \cite{PelRatproj}.
 This result points out   that many finer function-theoretic properties of  $A^p_\alpha$ do not carry over to
  $A^p_\om$ induced by $\om\in\I$.
\par Concerning rapidly decreasing weights, Dostanic \cite{D1} proved that the Bergman projection is bounded on $L^p_v$ only for $p=2$
in the case of Bergman spaces with the  exponential type weights
$w(r)=(1-r^2)^A\exp\left(\frac{-B}{(1-r^2)^\a}\right)$, $A\in\mathbb{R},\,B,\a>0$. The next result proves that it is
 a general phenomenon which holds for rapidly decreasing and smooth weights \cite{CP2,ZeyTams2012}.
 \begin{proposition}\label{th:1}
Assume that $\om(r)=e^{-2\phi(r)}$ is a radial weight such that $\phi$ is a positive $C^\infty$-function, $\phi'$ is positive on $[0,1)$, $\lim_{r\to 1^-}\phi(r)=\lim_{r\to 1^-}\phi'(r)=+\infty$  and
 \begin{equation}\label{bp4}\lim_{r\to 1^-} \frac{\phi^{ (n)}(r)}{\left(\phi'(r)\right)^n}=0, \quad\text{for any $n\in\N\setminus\{1\}$.
 }\end{equation}
 Then, the Bergman projection  is bounded from  $L^p_\om$ to $L^p_\om$ only for $p=2$.
\end{proposition}
\par Consequently, if  $\om$ a rapidly decreasing weight, it is natural to look for a subs\-ti\-tute  for the boundedness of the Bergman projection $P_\om$.
Inspired by the Fock space setting, the following result has been proved
for a canonical example \cite{CP2}.
\begin{theorem}\label{th:CP2}
Let $\om(r)=\exp\left(-\frac{\a}{1-r}\right)$,\,$\a>0$, and $1\le p <\infty$. Then, the Bergman projection $P_\om$
is bounded from $L^p_{\om^{p/2}}$ to $A^p_{\om^{p/2}}$.
\end{theorem}
The approach to prove this result relies on an instance of Schur's test and accurate  estimates for  the integral means of  order one of the
corresponding Bergman reproducing kernel \cite[Proposition $5$]{CP2}.
\begin{proposition}\label{pr:intmean}
Let $\om(r)=\exp\left(-\frac{\a}{1-r}\right)$,\,\,$\a>0$, and let $K(z)=\sum_{n=0}^\infty \frac{z^n}{2\om_{2n+1}}$. Then, there is a positive constant
$C$ such that
$$M_1(r,K)\asymp \frac{\exp\left(\frac{\a}{1-\sqrt{r}}\right)}{(1-r)^{\frac32}},\quad r\to 1^-.$$
\end{proposition}
 These estimates are obtained by using two key tools;  the sharp asymptotic estimates obtained in \cite{Kriete2003} for the moments
of the weight  in terms of the Legendre-Fenchel transform, and an upper estimate of $M_1(r,K)$
 by the $l^1$-norm of the $H^1$-norms of the Hadamard product of $K_r$ with certain smooth polynomials.
\par Finally, we mention  that a generalization  of Theorem \ref{th:1} for a class of not necessarily radial weights
has been achieved in \cite[Theorem $4.1$]{ArPau}. Their approach is different from that of \cite{CP2}, it uses
ideas from \cite{MarOrtJGA09}
and H\"{o}rmander-Berndtsson $L^2$-estimates for solutions of the $\bar{\partial}$-equation \cite{Berndtsoon01,Hormander90}.
\par We refer to \cite{D3,ZeyTurk13} for other results concerning the particular case $\om=v$ in \eqref{twoweight}.
\subsection{Two weight inequality}
 As it has been commented before, the weights $v$ satisfying \eqref{twoweight} when $\om$ is an standard weight and $1<p<\infty$,
 were characterized by Bekoll\'e and Bonami \cite{BB,B}. Recently \cite{PottRegueraJFA13}, it has been proved the following quantitative version of this result
 $$\|P^+_\a (f)\|_{L^p_{v(1-|z|^2)^\a}}\le C(p,\a)B_{p,\alpha}\left(v\right)\|f\|_{L^p_{v(1-|z|^2)^\a}},$$
 where $B_{p,\alpha}(v)$ was defined in \eqref{eq:BB}.
  With regard to the case $p=1$, we define  the weighted maximal function
    $$
    M_{\alpha}(\om)(z)=\sup_{z\in S(a)}\frac{\om(S(a))}{A_\a\left(S(a)
    \right)},\quad
    z\in\D.
    $$
\par It is known \cite{B,BekCan86} that
 the  weak $(1,1)$ inequality
    holds (and its analogue replacing $P_\a$ by $P^+_\a$)
     \begin{equation*}\label{weak11}
   \om\left(\left\{z\in\D: |P_\a(f)(z)|>\lambda\right\}\right)\le C_{\alpha,\om} \frac{||f||_{L^1_\om}}{\lambda}
    \end{equation*}
if and only if the weighted maximal function satisfies
$$M_{\alpha}(\om)(z)\le C\frac{\om(z)}{(1-|z|)^\a},\quad z\in\D.$$
\medskip\par
As far as we know, apart from Bekoll\'e-Bonami's results \cite{BB,B} on the standard Bergman projection $P_\a$ very little is known
about \eqref{twoweight} when $\om\neq v$. We note that \cite[Theorem $4.1$]{ArPau} or \cite[Theorem $1$]{CP2} may be seen as
 positive examples for \eqref{twoweight} in the context of rapidly increasing weights.
A recent result \cite{PelRatproj} describes  those regular weights $\om$ and $v$ for which \eqref{twoweight} holds for $1<p<\infty$.

\begin{theorem}\label{theorem:projections3}
Let $1<p<\infty$ and $\om,v\in\R$. Then the following conditions are equivalent:
\begin{itemize}
\item[\rm(a)] $P^+_\om:L^p_v\to L^p_v$ is bounded;
\item[\rm(b)] $P_\om:L^p_v\to L^p_v$ is bounded;
\item[\rm(c)]  $\displaystyle
\sup_{0<r<1} \frac{\widehat{v}(r)^{\frac{1}{p}}
    \left(\int_r^1\left(\frac{\om(s)}{v(s)}\right)^{p'}\,v(s)ds\right)^{\frac{1}{p'}}
    }{\widehat\om(r)}<\infty$;
\item[\rm(d)] $\displaystyle \sup_{0<r<1}\frac{\om(r)^p(1-r)^{p-1}}{v(r)}\int_0^r\frac{v(s)}{\om(s)^p(1-s)^p}\,ds<\infty$;
\item[\rm(e)] $\displaystyle \sup_{0<r<1}\left(\int_0^r\frac{v(s)}{\om(s)^p(1-s)^p}\,ds\right)^{\frac{1}{p}}
    \left(\int_r^1\left(\frac{\om(s)}{v(s)}\right)^{p'}\,v(s)ds\right)^{\frac{1}{p'}}<\infty;$
\item[\rm(f)] $\displaystyle
\sup_{0<r<1} \frac{\widehat{v}(r)^{\frac{1}{p}}
    \int_r^1\frac{\om(s)}{\left((1-s)v(s)\right)^{1/p}}\,ds
    }{\widehat\om(r)}<\infty$;
\item[\rm(g)]$\displaystyle \sup_{0<r<1}\frac{\om(r)(1-r)^{\frac{1}{p'}}}{v(r)^{1/p}}\int_0^r\frac{v(s)^\frac{1}{p}}{\om(s)(1-s)^{1+\frac{1}{p'}}}\,ds<\infty$.
\end{itemize}
\end{theorem}

It is worth noticing that condition (f) above makes sense also for $p=1$, and it turns out to be the  condition
 that describes those regular weights such that $P_\om$ is bounded on $L^1_v$ \cite{PelRatproj}.

\begin{theorem}\label{theorem:projections5}
Let $\om,v\in\R$. Then the following conditions are equivalent:
\begin{itemize}
\item[\rm(a)] $P_\om:L^1_v\to L^1_v$ is bounded;
\item[\rm(b)] $P^+_\om:L^1_v\to L^1_v$ is bounded;
\item[\rm(c)] $\displaystyle \sup_{0<r<1}\frac{\om(r)}{v(r)}\int_0^r\frac{\widehat{v}(s)}{\widehat{\om}(s)(1-s)}\,ds<\infty$;
\item[\rm(d)]$\displaystyle   \sup_{0<r<1}\frac{\widehat{v}(r)}{\widehat{\om}(r)}\int_r^1\frac{\om(s)}{v(s)(1-s)}\,ds<\infty$.
\end{itemize}
\end{theorem}

\end{document}